\documentclass[11pt]{article}

\usepackage{hyperref,amsmath,amssymb,latexsym,color,epsfig,enumerate,a4,dsfont,ulem,setspace,geometry,mathtools,amstext,amsthm,graphicx,caption,mathtools,mwe,amssymb,physics,enumitem}

\PassOptionsToPackage{normalem}{ulem}
\usepackage{ulem}
\theoremstyle{definition}
  \newtheorem*{defn*}{Definition}
\theoremstyle{plain}
  \newtheorem{prop}{Proposition}[section]
\theoremstyle{plain}
  \newtheorem{lem}{Lemma}[section]
\theoremstyle{plain}
  \newtheorem{thm}{Theorem} 
\theoremstyle{plain}
  \newtheorem*{thm*}{Theorem}  

\newtheorem{innercustomthm}{Theorem}
\newenvironment{customthm}[1]
  {\renewcommand\theinnercustomthm{#1}\innercustomthm}
  {\endinnercustomthm}

\ifx\proof\undefined\
   \newenvironment{proof}[1][\proofname]{\par
     \normalfont\topsep6\p@\@plus6\p@\relax
     \trivlist
     \itemindent\parindent
     \item[\hskip\labelsep
           \scshape
       #1]\ignorespaces
   }{%
     \endtrivlist\@endpefalse
   }
   \providecommand{\proofname}{Proof}
 \fi

\makeatother

\newcommand{\srbsquare}{\rotatebox{45}{\tiny{\ensuremath{\blacksquare}}}}
\newcommand{\asscheck}{\hskip.2em \null \hfill \srbsquare}

  {\small
    \begin{adjustwidth*}{1em}{1em}
      {\bf Verifying the assumptions of #1.}%
  }%
  {\asscheck\end{adjustwidth*}}
\title{Cycle lengths in expanding graphs}

\author{
  Limor Friedman
  \thanks{
    School of Mathematical Sciences, Raymond and Beverly Sackler Faculty of Exact Sciences, Tel Aviv University, Tel Aviv,
    6997801, Israel. Email: limorf1@mail.tau.ac.il.
  }
  \and
  Michael Krivelevich
  \thanks{
    School of Mathematical Sciences, Raymond and Beverly Sackler Faculty of Exact Sciences, Tel Aviv University, Tel Aviv,
    6997801, Israel. Email: krivelev@tauex.tau.ac.il.  Research supported in part by USA-Israel BSF grant 2018267, and by ISF grant 1261/17.
  }
}
\date{}

\begin{document}
\maketitle
\begin{abstract}
For a positive constant $\alpha$ a graph $G$ on $n$ vertices is called an $\alpha$-expander if every vertex set $U$ of size at most $n/2$ has an external neighborhood whose size is at least $\alpha\left|U\right|$. We study cycle lengths in expanding graphs. We first prove that cycle lengths in $\alpha$-expanders are well distributed. Specifically, we show that for every $0<\alpha\leq1$ there exist positive constants $n_{0}$, $C$ and $A=O(1/\alpha)$ such that for every $\alpha$-expander $G$ on $n\geq n_{0}$ vertices and every integer $\ell\in\left[C\log n,\frac{n}{C}\right]$, $G$ contains a cycle whose length is between $\ell$ and $\ell+A$; the order of dependence of the additive error term $A$ on $\alpha$ is optimal. Secondly, we show that every $\alpha$-expander on $n$ vertices contains $\Omega\left(\frac{\alpha^{3}}{\log\left(1/\alpha\right)}\right)n$ different cycle lengths. Finally, we introduce another expansion-type property, guaranteeing the existence of a linearly long interval in the set of cycle lengths. For $\beta>0$ a graph $G$ on $n$ vertices is called a $\beta$-graph if every pair of disjoint sets of size at least $\beta n$ are connected by an edge. We prove that for every $\beta <1/20$ there exist positive constants $b_{1}=O\left(\frac{1}{\log\left(1/\beta\right)}\right)$
and $b_{2}=O\left(\beta\right)$ such that every $\beta$-graph $G$ on $n$ vertices contains a cycle of length $\ell$ for every integer $\ell\in\left[b_{1}\log n,(1-b_{2})n\right]$; the order of dependence of $b_{1}$ and $b_{2}$ on $\beta$ is optimal. 
\end{abstract}

\section{Introduction}

Intuitively, a graph $G$ is an expander if every vertex set $U$ expands outside substantially, meaning, has an external neighborhood whose size is comparable to $\left|U\right|$. Expanders are one of the central notions in modern graph theory. The reader is encouraged to consult the survey of Hoory, Linial and Wigderson \cite{HLW}, devoted entirely to expanding graphs and covering many aspects of this subject. More formally, for a graph $G=\left(V,E\right)$ and a vertex set $U\subseteq V$ we denote by $N_{G}(U)$ the external neighborhood of $U$ in $G$, that is, $N_{G}(U)=\left\{ v\in V\backslash U:\text{\ensuremath{v} has a neighbor in \ensuremath{U}}\right\}$.
In this work we adopt the following definition of an expander:
\begin{defn*}
Let $G=\left(V,E\right)$ be a graph on $n$ vertices, and let $\alpha>0$. The graph $G$ is an \textit{$\alpha$-expander} if $\left|N_{G}\left(U\right)\right|\geq\alpha\left|U\right|$ for every vertex set $U\subseteq V$ satisfying $\left|U\right|\leq \left\lceil n/2\right\rceil$. 
\end{defn*}
This notion of expansion is fairly common and has been utilized in several places, see, e.g.\ \cite{A}, or \cite[Chapter 9]{AS}. In this work we think of $\alpha$ as a small constant and consider the behavior of other parameters as a function of $\alpha$. While it does not aim to capture or to reflect the strongest possible level of expansion, it is strong enough to derive many nice graph properties as can be seen in recent survey by the second author \cite{K}. For example, it is easy to see that every $\alpha$-expander is connected and has logarithmic diameter. This notion is also very natural as such expanders are omnipresent \textendash{} they typically appear in supercritical random graphs (see \cite{Kr}), and in graphs without small separators (Proposition 5.3 in \cite{K}). One important class of such expanders is $\left(n,d,\lambda\right)$-graphs; these are $d$-regular graphs
on $n$ vertices with second largest eigenvalue in absolute value $\lambda$. An $\left(n,d,\lambda\right)$-graph is an $\alpha$-expander with $\alpha=\frac{d-\lambda}{2d}$ (see e.g.\ Corollary~9.2.2 of \cite{AS}).

In this work we study cycle lengths in expanding graphs. The study of cycle lengths in graphs with certain properties has long been fundamental (see e.g., \cite{B,EFRS,F,GHLM,GKS,KSV,LM,MS,SV}). 
It is known (see for example Corollary~3.2 and Theorem~7.4 in \cite{K}) that $\alpha$-expanders on $n$ vertices have a cycle of length $\Omega_{\alpha}\left(n\right)$ and a cycle as short as $O_{\alpha}\left(\log n\right)$. A longest cycle of length linear in $n$ is obviously optimal. As for a shortest cycle, there are examples of strong expanders with girth logarithmic in $n$ (e.g.\ the construction of Ramanujan graphs by Lubotzky, Phillips and Sarnak  \cite{LPS}). What can be said then about cycle lengths in $G$ between the two extremes?

For a graph $G$ let $L\left(G\right)$ be the set of cycle lengths in $G$, i.e.\ $$L\left(G\right)=\left\{ \ell:\ \text{\ensuremath{G} contains a cycle of length \ensuremath{\ell}}\right\}.$$ Perhaps the most natural questions one can ask, apart from the two extremes discussed above, are: what is the size of $L\left(G\right)$? How is $L\left(G\right)$ distributed?\ Does $L\left(G\right)$ include complete intervals? 
\\Unfortunately, one cannot hope to find a complete interval of non-trivial length in the set of cycle lengths of every $\alpha$-expander. Indeed, the complete bipartite graph $K_{\frac{\alpha}{2}n,\left(1-\frac{\alpha}{2}\right)n}$ is an $\alpha$-expander with no odd cycles. Another important example is the following way of producing new expanders by stretching edges of a bounded degree expander:
\begin{prop}
(\cite{K}) Let $G$ be an $\alpha$-expander of maximum degree $\Delta=O\left(1\right)$,
and let $m$ be a positive integer. Subdividing each edge of $G$
$m$ times produces an $\Omega_{\Delta}\left(\alpha/m\right)$-expander
$G'$.
\end{prop}
In particular the length of every cycle in the obtained graph $G'$ is divisible by $m+1$,
implying that every two cycle lengths in $L\left(G'\right)$ are at
the distance at least $m+1$ from each other.

\subsection*{Our results}
\label{subsec:res}
\addcontentsline{toc}{subsection}{\nameref{subsec:res}}
Our main result shows that the set of cycle lengths of an $\alpha$-expander
is well spread:
\begin{thm}
For every $0<\alpha\leq1$ there exist positive constants $A=O\left(\frac{1}{\alpha}\right),\ a_{1}=O\left(\frac{1}{\alpha}\right)$
and $a_{2}=2^{-O\left(\frac{\log\left(1/\alpha\right)}{\alpha}\right)}$
such that for large enough $n$, every $\alpha$-expander $G$ on
$n$ vertices contains a cycle whose length is between $\ell$ and
$\ell+A$, for every integer $\ell\in\left[a_{1}\log n,a_{2}n\right]$.
\end{thm}
Going back to the proposition, by stretching the edge of a bounded degree expander one can obtain an $\alpha$-expander $G$ such that every element in $L\left(G\right)$ is divisible by $\Theta\left(\frac{1}{\alpha}\right)$. This implies that the order of dependence of the constant $A$ on $\alpha$
in Theorem 1 is optimal.

From Theorem $1$, for every $\alpha$-expander $G$ on $n$ vertices
the size of $L\left(G\right)$ is linear in $n$, but the dependence
on $\alpha$ given by the theorem is far from the truth. In fact,
we have a polynomial dependence, as demonstrated by the next
theorem: 
\begin{thm}
Let $0<\alpha\leq1$ and let $G$ be an $\alpha$-expander on $n$ vertices, then $\left|L\left(G\right)\right|=\Omega\left(\frac{\alpha^{3}}{\log\left(1/\alpha\right)}\right)n$.
\end{thm}
As we mentioned, for an $\alpha$-expander $G$ the set $L\left(G\right)$ might not contain any interval of non-trivial length. In the last section we analyze a different expansion-type property that guarantees the existence of a non-trivial interval in the set of cycle lengths:
\begin{defn*}
Let $G=\left(V,E\right)$ be a graph on $n$ vertices, and let $\beta>0$.
The graph $G$ is a {\it$\beta$-graph} if every pair of disjoint vertex
sets $A,B\subseteq V$ of sizes $\left|A\right|,\left|B\right|\geq\beta n$ are connected by an edge.
\end{defn*}
Note that if $G$ is a $\beta$-graph on $n$ vertices then for every vertex set $U$ of size $\left|U\right|\geq\beta n$ we have $\left|N_{G}\left(U\right)\right|>n-\left|U\right|-\beta n$. Similarly to $\alpha$-expanders, $\beta$-graphs can be obtained through spectral conditions, from the expander mixing lemma, an $\left(n,d,\lambda\right)$-graph with $\frac{d}{\left|\lambda\right|}\geq\frac{1}{\beta^{2}}$  is a $\beta$-graph. 

In \cite{HKS} Hefetz, Krivelevich and  Szab\'o proved that for large enough $n$ and $\beta=O\left(\frac{\log\log n}{\log n}\right)$, every $\beta$-graph $G$ on $n$ vertices contains a cycle of length $\ell$ for every integer $\frac{8\beta n\log n}{\log\log n}\leq \ell\leq \left(1-3\beta\right)n$.
The upper bound is tight as shown by a disjoint union of $K_{n+1-\beta n}$ and $\beta n-1$ isolated vertices. The authors conjectured that the lower bound can be improved to $\frac{c\log n}{\log\left(1/\beta\right)}$ for some constant $c$.
We prove this is indeed the case, as asserted by the following theorem:
\begin{thm}
For every $0<\beta<\frac{1}{20}$ there exist positive constants $b_{1}=O\left(\frac{1}{\log\left(1/\beta\right)}\right)$
and $b_{2}=O\left(\beta\right)$, s.t.\ every $\beta$-graph $G$ on $n$ vertices contains a cycle of length $\ell$ for every integer  $\ell\in\left[b_{1}\log n,(1-b_{2})n\right]$.  
\end{thm}
Note that in Theorem~3 we only require $\beta=O(1)$ while in \cite{HKS}  $\beta$ tends to zero, and even at prescribed explicit rate, when $n$ tends to infinity.
We view this as a substantial improvement in comparison with \cite{HKS}. As we have already mentioned, the upper bound is tight. The lower bound is tight as well. For example the Ramanujan graphs by Lubotzky, Phillips and Sarnak \cite{LPS} are $\beta$-graphs with girth at least $\frac{\log n}{4\log(1/\beta)}$.

The rest is structured as follows: in Section 2 we prove Theorem 1, in Section 3 we prove Theorem 2 and in Section 4 we prove Theorem 3. 

\subsection*{Notation and terminology}
\label{subsec:not}
\addcontentsline{toc}{subsection}{\nameref{subsec:not}}

Our notation is mostly standard. As stated before,
for a graph $G=\left(V,E\right)$ and a vertex subset $U\subseteq V$
we denote by $N_{G}\left(U\right)$ the external neighborhood of $U$
in $G$. For a positive integer $r$ we denote by $B_{G}\left(U,r\right)$
the ball of radius $r$ around $U$ in $G$, i.e.\, $B_{G}\left(U,r\right)=\left\{ v\in V:\text{dist}_{G}\left(U,v\right)\leq r\right\} $. 
By the length of a path we mean the number of edges in the path.
For a rooted tree $T$ with levels $L_{1},...,L_{k}$, for $1\leq j_{1}\leq j_{2}\leq k$
we denote by $T_{\left[j_{1},j_{2}\right]}$ the union $\bigcup_{i=j_{1}}^{j_{2}}L_{i}$.
For every vertex $v\in V\left(T\right)$ we denote by $T_{v}$ the
subtree of $T$ rooted in $v$. 
All logarithms are in base $2$ unless stated otherwise. We use the following bounds: $\log x\leq x$ for
$x\geq1$ and $\log\left(1+x\right)\geq x/2$ for $0\leq x\leq1$.
We suppress the rounding notation occasionally to simplify the presentation.

\section{Cycle lengths are well distributed}

In this section we prove our main result, Theorem 1, which shows that the set of cycle lengths of an $\alpha$-expander is well spread. Before presenting the proof some preparation is needed. Throughout this section we sometimes use a more general definition of an expander: 
\begin{defn*}
Let $G=\left(V,E\right)$ be a graph, let $k>0$ and let $\alpha>0$.
The graph $G$ is a $\left(k,\alpha\right)$-expander if $\left|N_{G}\left(U\right)\right|\geq\alpha\left|U\right|$
for every vertex subset $U\subseteq V$ satisfying $\left|U\right|\leq k$.
\end{defn*}

\subsection{Auxiliary results}
\label{subsec:ar}
\addcontentsline{toc}{subsection}{\nameref{subsec:ar}}

 We start with several basic lemmas (in the order in which they appear in the proof of Theorem 1).
 
 \begin{lem}
Let $G=(V,E)$ be a $(k,\alpha)$-expander on $n$ vertices and let $C$ be a connected component of $G$. Then for $G'=G[C]$ we have $\text{diam}\left(G'\right)<(\left\lceil \frac{n}{k}\right\rceil -1)(2\left\lceil \frac{\log k}{\log\left(1+\alpha\right)}\right\rceil +1)$. 
\vspace{1mm}
\\In particular, if $k=\delta n$ for some $0<\delta<1$ and $0<\alpha\leq1$, then $\text{diam}\left(G'\right)\leq \frac{7}{\delta\alpha}\log n$. 
\end{lem}

\begin{proof}
Since $G$ is a $(k,\alpha)$-expander for every non-negative integer $r$ and every vertex $v\in V$ we have $|B_{G}(v,r)|\geq \min \{k,(1+\alpha)^{r}\}$ (see, e.g.\, Proposition~3.1 in \cite{K}). Hence for $r_{0}=\left\lceil \frac{\log k}{\log\left(1+\alpha\right)}\right\rceil$ the ball  $B_{G}(v,r_{0})$ contains at least $k$ vertices. \\
Let $u,v\in C$ and let $P$ be a path of shortest length between them in $G$. Note that if $w_1, w_2\in V(P)$ are more than $2r_{0}$ apart along $P$ the balls $B_{G}(w_1,r_{0})$ and $B_{G}(w_2,r_{0})$ are disjoint. Thus the length of $P$ must be less than $(\left\lceil \frac{n}{k}\right\rceil-1) (2r_{0}+1)$, as otherwise we get $n/k$ pairwise-disjoint balls of volume $k$ in $G$. 
\end{proof}

\begin{lem}
Let $0<\alpha\leq 1$. Let $G=\left(V,E\right)$ be a $(k,\alpha)$-expander on $n$ vertices and let $V_{0}\subseteq V$ be a vertex set of size $\left|V_{0}\right|\leq\epsilon n$, where $0<\epsilon\leq\frac{\alpha^{2}k}{8n}$. Then the graph $G\left[V\backslash V_{0}\right]$ contains a vertex set $U\subseteq V\backslash V_{0}$ of size $\left|U\right|\geq\left(1-\frac{3\epsilon}{\alpha}\right)n$, such that the graph $G\left[U\right]$ is a $\left(k,\frac{\alpha}{2}\right)$-expander. 
\end{lem}
\begin{proof}
Start with $G'=G\left[V\backslash V_{0}\right]$, and as long as there is a vertex subset $A\subseteq V\left(G'\right)$ of size $\left|A\right|\leq k$ satisfying $\left|N_{G'}\left(A\right)\right|<\frac{\alpha}{2}\left|A\right|$, delete $A$ and update $G'\coloneqq G'\left[V\left(G'\right)\backslash A\right]$.
Denote by $Z$ the disjoint union of the deleted sets and note that $Z$ satisfies $\left|N_{G}\left(Z\right)\backslash V_{0}\right|<\frac{\alpha}{2}\left|Z\right|$.
\\While $\left|Z\right|\leq k$, we must have $\left|Z\right|<2\epsilon n/\alpha$. Indeed, since $G$ is a $(k,\alpha)$-expander
$$\alpha\left|Z\right|-\epsilon n\leq|N_G(Z)|-|V_0|\leq \left|N_{G}\left(Z\right)\backslash V_{0}\right|<\frac{\alpha}{2}\left|Z\right|.$$
Assume $\left|Z\right|>k$ after some iteration, and let $A$ be the set deleted at this iteration. We have
$$
\alpha\left|A\right|\leq\left|N_{G}\left(A\right)\right|\leq\left|Z\backslash A\right|+\left|V_{0}\right|+\left|N_{G}\left(A\right)\backslash\left(V_{0}\cup Z\right)\right|<\frac{2\epsilon n}{\alpha}+\epsilon n+\frac{\alpha}{2}\left|A\right|\leq\frac{3\epsilon n}{\alpha}+\frac{\alpha}{2}\left|A\right|,
$$
implying $\left|A\right|<6\epsilon n/\alpha^{2}$. But then we get
$$
k<\left|Z\right|=\left|Z\backslash A\right|+\left|A\right|<\frac{2\epsilon n}{\alpha}+\frac{6\epsilon n}{\alpha^{2}}\leq\frac{8\epsilon n}{\alpha^{2}}\leq k
$$
\textemdash{} a contradiction. \\This indicates that in the end of the deletion process $\left|Z\right|<2\epsilon n/\alpha$,
and therefore the final graph $G'$ satisfies the requirements with
$U=V\backslash\left(V_{0}\cup Z\right)$. 
\end{proof}


 \begin{lem}
 Let $G=(V,E)$ be an $\alpha$-expander and let $A,B\subseteq V$ be two vertex sets of sizes $|A|,|B|\geq t$ for some $t>0$. Then $G$ contains at least $\frac{t\alpha}{1+\alpha}$ vertex-disjoint paths between $A$ and $B$.
 \end{lem}
 
 \begin{proof}
 Let $\mathcal{P}$ be a maximal family of vertex-disjoint paths between $A$ and $B$ in $G$. We claim that $\mathcal{P}$ is of size at least $\frac{t\alpha}{1+\alpha}$. By Menger's theorem it is enough to show that every vertex set separating $A$ from $B$, meaning a set $C$ such that there are no paths between $A\backslash C$ and $B\backslash C$ in $G\backslash C$, is of size at least $\frac{t\alpha}{1+\alpha}$. Let $C\subseteq V$ be such a separating set and let $A'$ be the union of the connected components containing $A\backslash C$ in $G[V\backslash C]$, note that $N_{G}(A')\subseteq C$. We may assume w.l.o.g. that $\left|A'\right|\leq |V|/2$ (otherwise we can look at the union of the connected components containing $B\backslash C$ which is contained in $V\backslash A'$). Due to $\alpha$-expansion of $G$, we get
 $$\alpha t-\alpha\left|C\right|\leq \alpha\left(\left|A\right|-\left|C\right|\right)\leq \alpha\left|A\backslash C\right|\leq 
 \alpha\left|A'\right|\leq  \left|N_{G}\left(A'\right)\right|\leq \left|C\right|$$
implying $\left|C\right|\geq\frac{t\alpha}{1+\alpha}$. 
\end{proof}

\begin{lem}
Let $G=\left(V,E\right)$ be a $\left(k,\alpha\right)$-expander on $n$ vertices. Let $C>0$ and let $V_{1},...,V_{r}\subseteq V$ be disjoint non-empty vertex sets, each of size at most $C$, such that $V=\bigcup_{i=1}^{r}V_{i}$. Contracting each $V_{i}$ to a vertex creates a graph $G'$ on $r\geq\frac{n}{C}$ vertices, which is a $\left(\frac{k}{C},\frac{\alpha}{C}\right)$-expander. 
\end{lem}

\begin{proof}
For every $1\leq i\leq r$ denote by $v_{i}\in V\left(G'\right)$ the vertex obtained by contracting the set $V_{i}$. Let $U=\left\{ v_{i_{1}},...,v_{i_{m}}\right\} \subseteq V\left(G'\right)$ be a set of size $|U|\leq k/C$. Denote by $W$ the set of all vertices in $V$ corresponding to $U$, i.e.\ $W=\bigcup_{j=1}^{m}V_{i_{j}}$.
Note that $|U|\leq\left|W\right|\leq C|U|\leq k$ and that $\left|N_{G}\left(W\right)\right|\leq C\left|N_{G'}\left(U\right)\right|$.
Since $G$ is a $\left(k,\alpha\right)$-expander we get 
$$\left|N_{G'}\left(U\right)\right|\geq\frac{1}{C}\left|N_{G}\left(W\right)\right|\geq\frac{\alpha}{C}\left|W\right|\geq\frac{\alpha}{C}\left|U\right|.$$
We conclude that $G'$ has the required expansion property.
\end{proof}

\begin{lem}
Let $k$ and $\ell$ be positive integers. Assume that $G=(V,E)$ is a graph on more than $k$ vertices, in which every vertex set $W\subseteq V$ of size $\left|W\right|=k$ satisfies $\left|N_{G}\left(W\right)\right|\geq\ell$.
Let $v\in V$ be such that the connected component of
$v$ contains at least $k$ vertices, then $G$ contains a path of
length $\ell$ starting from $v$. 
\\In particular, in a $\left(k,\alpha\right)$-expander for every vertex there is a path of length $\left\lceil \alpha\left\lfloor k\right\rfloor \right\rceil $ starting from it.
\end{lem}
\begin{proof}
See, e.g.\, Proposition~2.1 in \cite{Kri}.
\end{proof}


Our key lemma requires yet another lemma which is similar in spirit to Lemma~2.2:
\begin{lem}
Let $G=\left(V,E\right)$ be an $\alpha$-expander on $n$ vertices and let $W\subseteq V$ be a vertex set such that $\left|W\right|>n/2$ and $\left|N_{G}\left(W\right)\right|\leq\alpha\epsilon n$
for $0<\epsilon<1/4$. Then the graph $G\left[W\right]$ contains
a vertex set $U\subseteq W$ of size $\left|U\right|>\left(\frac{1}{2}-2\epsilon\right)n$,
such that $G\left[U\right]$ is a $\left(\left(\frac{1}{2}-2\epsilon\right)n,\frac{\alpha}{2}\right)$-expander. 
\end{lem}
\begin{proof}
Start with $G'=G\left[W\right]$, and as long as there is a vertex
subset $A\subseteq V\left(G'\right)$ of size $\left|A\right|\leq\left(\frac{1}{2}-2\epsilon\right)n$
satisfying $\left|N_{G'}\left(A\right)\right|<\frac{\alpha}{2}\left|A\right|$,
delete $A$ and update $G'\coloneqq G'\left[V\left(G'\right)\backslash A\right]$.\\ Denote by $Z$ the disjoint union of the deleted sets and note that $Z$ satisfies $\left|N_{G}\left(Z\right)\cap W\right|<\frac{\alpha}{2}\left|Z\right|$.\\
Assume $\left|Z\right|>2\epsilon n$ after some iteration, and let
$A$ be the set deleted at this iteration. Since $\left|Z\right|=\left|Z\backslash A\right|+\left|A\right|\leq n/2$
we have
\[
\alpha\left|Z\right|\leq\left|N_{G}\left(Z\right)\right|\leq\left|N_{G}\left(Z\right)\cap W\right|+\left|N_{G}\left(W\right)\right|<\frac{\alpha}{2}\left|Z\right|+\alpha\epsilon n<\alpha\left|Z\right|
\]
\textemdash{} a contradiction. \\This indicates that in the end of the deletion process $\left|Z\right|\leq2\epsilon n$,
and therefore the final graph $G'$ satisfies the requirements with $U=W\backslash Z$.
\end{proof}



\subsection{Statement and proof of the main lemma}
\label{subsec:key}
\addcontentsline{toc}{subsection}{\nameref{subsec:key}}
\begin{lem}
For every $0<\alpha\leq1$ there exist positive constants $\Delta =O(1/\alpha^{5})$, $C_{0}=O(1/\alpha)$, $C_{1}=O(1/\alpha^{5})$ and $C_2=O(1/\alpha)$ such that for every $\alpha$-expander $G$ on $n$ vertices and for every vertex $v_{0}\in V\left(G\right)$, $G$ contains a tree $T$ rooted at $v_{0}$,
with levels $L_{1}=\left\{ v_{0}\right\},...,L_{k_{0}},...,L_{k_{1}},...,L_{k_{2}}$
satisfying:
\begin{enumerate}
    \item $k_{0}\leq C_{0}\log n$, $k_{1}-k_{0}\leq C_{1}$, $k_{2}-k_{1}\leq C_{2}$; 
    \item $k_{0}$ is the first index such that $\left|T_{\left[1,k_{0}\right]}\right|\geq \alpha^4 n/200$;
    \item The degree in $T$ of every $v\in T_{\left[k_{1},k_{2}\right]}$ is at most $\Delta$;
    \item $T_{\left[k_{1},k_{2}\right]}$ contains an $\left(\frac{n}{10},\frac{\alpha}{5}\right)$-expander on at least $n/10$ vertices.
\end{enumerate}
\end{lem}

\begin{proof}
Let $\sigma$ be an arbitrary ordering of $V\left(G\right)$ starting
from $v_{0}$. Run the BFS algorithm on $G$ according to $\sigma$
until a tree of size exactly $\left\lceil \alpha^{4} n/200\right\rceil $
is obtained, denote it by $T'$; this is possible since $G$ is connected.
Suppose we stopped while adding the neighbors of $v\in V\left(G\right)$.
We shall extend $T'$ by first adding all the neighbors of $v$ lying
outside $T'$ and then continue running the BFS algorithm with the
restriction that the degree of every newly explored vertex is at most $\Delta\coloneqq\left\lceil 1600/\alpha^{5}\right\rceil $
(connecting every vertex to the first $\Delta-1$ neighbors it has
outside the current tree according to $\sigma$). Denote the final
tree by $T$ and its levels by $L_{1}=\left\{v_{0}\right\},...,L_{k}$.

Denote by $k_{0}$ the number of levels in $T'$, i.e.\ it is the first index such that $\left|T_{\left[1.k_{0}\right]}\right|\geq\alpha^4 n/200$. Since $T'$ is a subtree of a BFS tree, by applying Lemma~2.1 to $G$ (with $\delta=1/2$ and $C=V(G)$) we get that $k_0\leq\text{diam}(G)+1\leq 15 \log n/\alpha \eqqcolon C_0\log n$. 
\\ Now, observe that $T$ contains at least $n/2$ vertices. Indeed, let $X\subseteq V\left(T\right)$ be the set of all the vertices whose degree in $T$ is at least $\Delta$.
Note that $\left|X\right|\leq 2n/\Delta\leq\alpha^{5}n/800$,
and that the neighborhood of every vertex in $T\backslash X$ is contained
in $T$. Let $G'=G\left[V\left(G\right)\backslash X\right]$, due
to $\alpha$-expansion of $G$ every vertex set $U\subseteq V\left(G'\right)$
of size $\frac{2}{\alpha}\left|X\right|\leq\left|U\right|\leq n/2$
satisfies
\[
\left|N_{G'}\left(U\right)\right|\geq\left|N_{G}\left(U\right)\right|-\left|X\right|\geq \alpha |U|-|X|\geq\frac{\alpha}{2}\left|U\right|.
\]
Suppose $T$ contains less than $n/2$ vertices, then $\frac{2}{\alpha}\left|X\right|\leq\alpha^4 n/400\leq\left|T\backslash X\right|\leq n/2$,
implying there is a vertex in $T\backslash X$ adjacent to a vertex
outside of $T$ \textemdash{} a contradiction.\\
Let $t_{1}$ be the first index such that $\left|T_{\left[1,t_{1}\right]}\right|\geq n/4$,
and let $t_{2}$ be the first index such that $\left|T_{\left[1,t_{2}\right]}\right|\geq n/2$.
We claim that that $t_{2}-t_{1}=O\left(\frac{1}{\alpha}\right)$. If $t_{2}-t_{1}\leq1$
we are done, otherwise $k_{0}<t_{1}+1<t_{2}$. Observe that for every
$1\leq i<k$ we have $L_{i+1}\supseteq N_{G}\left(T_{\left[1,i\right]}\backslash X\right)\backslash X.$
For every $k_{0}\leq i\leq t_{2}-1$ we have $\frac{4}{\alpha}\left|X\right|\leq \alpha^4n/200\leq\left|T_{\left[1,i\right]}\right|\leq n/2$,
and since $G$ is an $\alpha$-expander we get
\[
\left|L_{i+1}\right|\geq\left|N_{G}\left(T_{\left[1,i\right]}\backslash X\right)\backslash X\right|\geq\alpha\left|T_{\left[1,i\right]}\right|-2\left|X\right|\geq\frac{\alpha}{2}\left|T_{\left[1,i\right]}\right|.
\]
Therefore, for every $k_{0}\leq i\leq t_{2}-1$ we have  $\left|T_{\left[1,i+1\right]}\right|\geq\left(1+\frac{\alpha}{2}\right)\left|T_{\left[1,i\right]}\right|$,
hence
\[
\frac{n}{2}\geq\left|T_{\left[1,t_{2}-1\right]}\right|\geq\left(1+\frac{\alpha}{2}\right)^{t_{2}-t_{1}-1}\left|T_{\left[1,t_{1}\right]}\right|\geq\left(1+\frac{\alpha}{2}\right)^{t_{2}-t_{1}-1}\frac{n}{4},
\]
implying $t_{2}-t_{1}\leq 1/\log\left(1+\frac{\alpha}{2}\right)+1\leq 5/\alpha$.

Let $k_{1}=\max\left\{ i<t_{1}\ :\ \left|L_{i}\right|<\alpha n/12\right\} +1$. Note that $k_{1}\geq k_{0}$ (otherwise $\alpha^4 n/200>\left|T_{\left[1,k_{1}\right]}\right|\geq\left|L_{k_{1}}\right|\geq\alpha n/12$). Either $k_{1}=k_{0}$ or we have 
\[
\frac{n}{4}\geq\left|T_{\left[1,t_{1}-1\right]}\right|\geq\left|T_{\left[1,k_{1}-1\right]}\right|\geq\left(1+\frac{\alpha}{2}\right)^{k_{1}-k_{0}-1}\left|T_{\left[1,k_{0}\right]}\right|\geq\left(1+\frac{\alpha}{2}\right)^{k_{1}-k_{0}-1} \frac{\alpha^4n}{200},
\]
implying $k_{1}-k_{0}\leq\log\left(\frac{50}{\alpha^4}\right)/\log\left(1+\frac{\alpha}{2}\right)+1\leq 201/\alpha^5\eqqcolon C_1$. 

Let $k_{2}=\min\left\{ i>t_{2}\ :\ \left|L_{i}\right|<\alpha n/12\right\} -1$,
if such minimum does not exist take $k_{2}=k$ (the last level of
$T$). Observe that either $k_{1}=t_1$ or 
\[
\frac{n}{2}\geq\left|T_{\left[k_{1},t_{1}-1\right]}\right|\geq\frac{\alpha n}{12}\left(t_{1}-k_{1}\right),
\]
hence $t_{1}-k_{1}\leq 6/\alpha$. Similarly, $k_{2}-t_{2}\leq\frac{6}{\alpha}$,
and we get that $k_{2}-k_{1}\leq 17/\alpha\eqqcolon C_2$.
\begin{figure}[h]
\centering
\includegraphics[scale=0.8]{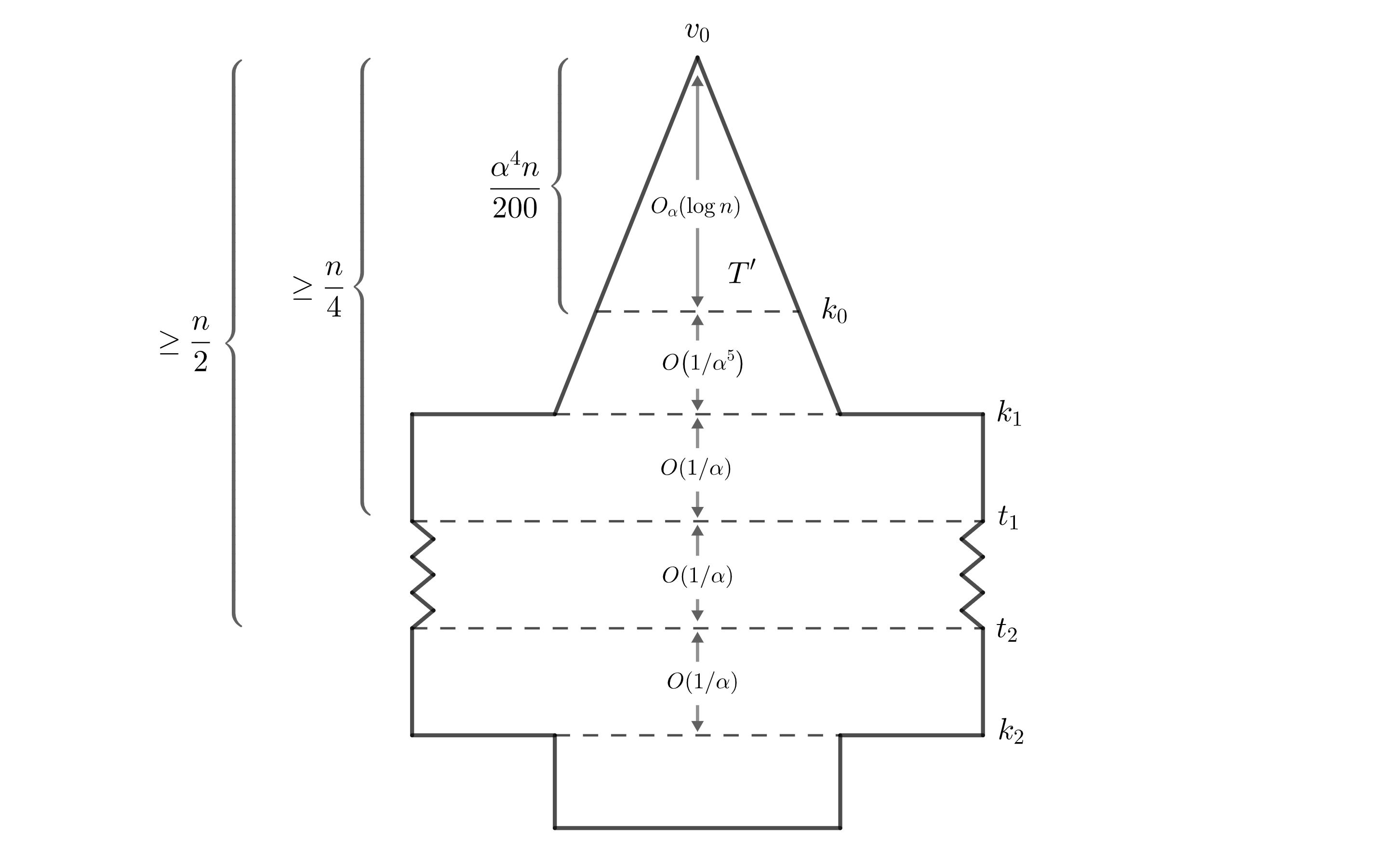}
\caption{The tree $T$ obtained in Lemma~2.7}
\end{figure}


Last, we show that $T_{\left[k_{1},k_{2}\right]}$ contains an $\left(\frac{n}{10},\frac{\alpha}{2}\right)$
-expander on at least $n/10$ vertices. Indeed, denote $W=T_{\left[k_{1},k_{2}\right]}\backslash X$
and observe that 
\[
\left|W\right|\geq\left|T_{\left[1,k_{2}\right]}\right|-\left|T_{\left[1,k_{1}-1\right]}\right|-\left|X\right|\geq\frac{n}{2}-\frac{n}{4}-\frac{n}{48}>\frac{n}{5}.
\]
From the construction of $T$ it follows that $N_{G}\left(W\right)\subseteq L_{k_{1}-1}\cup L_{k_{2}+1}\cup X$
(where $L_{k_{2}+1}=\emptyset$ if $k_{2}=k$), thus $$\left|N_{G}\left(W\right)\right|\leq| L_{k_{1}-1}|+|L_{k_{2}+1}|+|X|\leq2\cdot \frac{\alpha n}{12}+\frac{\alpha^5n}{800}\leq\frac{\alpha n}{5}.$$
Since $G$ is an $\alpha$-expander we must have $\left|W\right|>n/2$,
from Lemma~$2.6$ we get that $G\left[W\right]$ contains the required
expander.
\end{proof}



\subsection{Proof of Theorem 1}
\label{subsec:t1}
\addcontentsline{toc}{subsection}{\nameref{subsec:t1}}

We are now ready to  restate and prove Theorem 1. 
\begin{customthm}{1}
For every $0<\alpha\leq 1$ there exist positive constants $A=O\left(\frac{1}{\alpha}\right),\ a_{1}=O\left(\frac{1}{\alpha}\right)$
and $a_{2}=2^{-O\left(\frac{\log\left(1/\alpha\right)}{\alpha}\right)}$
such that for large enough $n$, every $\alpha$-expander $G$ on
$n$ vertices contains a cycle whose length is between $\ell$ and
$\ell+A$, for every integer $\ell\in\left[a_{1}\log n,a_{2}n\right]$.
\end{customthm}
\begin{proof}
Run the BFS algorithm on $G$ using an arbitrary ordering of $V(G)$ until a tree $T$ of size exactly $\left\lceil \alpha^2n/16\right\rceil$ is obtained; this is possible since $G$ is an $\alpha$-expander and is therefore connected. Since $T$ is a subtree of a BFS tree, by applying Lemma 2.1 to $G$ (with $\delta=1/2$ and $C=V(G)$) it follows that the number of levels in $T$ is at most $\text{diam}(G)+1\leq 15 \log n/\alpha$. 
\\From Lemma 2.2, the graph $G\left[V\backslash V\left(T\right)\right]$ contains a vertex set $U_{1}$ of size $|U_1|=n'\geq\left(1-\frac{3\alpha}{16}\right)n$ such that $G_1=G\left[U_{1}\right]$ is an  $\frac{\alpha}{2}$-expander. Denote by $Z$ the set $V\left(G\right)\backslash\left(V\left(T\right)\cup U_{1}\right)$ and note that $\left|Z\right|\leq 3\alpha n/16$. While $T$
has at least $\alpha^3n/32$ neighbors in $Z$ remove them from $Z$ and add them to $T$ by connecting each one to one of its neighbors in $T$. This entire process takes at most $\left|Z\right|/\left(\alpha^3n/32\right)\leq 6/\alpha^2$ iterations. Since in each iteration we add at most one new level to $T$, the number of levels in the current tree $T$ is at most $15\log n/\alpha+6/\alpha^2$.
\\Since $G$ is an $\alpha$-expander, $\left|N_{G}\left(T\right)\right|\geq\alpha\left|T\right|\geq\alpha^3n/16$. Let $X_{0}$ be the set of neighbors of $T$ in $U_{1}$, note that $\left|X_{0}\right|\geq\alpha^3n/32$ as $T$ has at most $\alpha^3n/32$ neighbors in $Z$.
\\Pick a vertex $y\in X_{0}$ and apply Lemma $2.7$ to $G_1$ with $v_{0}=y$. We obtain a tree $T'$ with levels $L_{1}=\left\{ y\right\} ,...,L_{k_{0}},...L_{k_{1}},...,L_{k_{2}}$
satisfying the properties in Lemma 2.7 with $\Delta=\Delta\left(\alpha/2\right)$,
$C_{0}=C_{0}\left(\alpha/2\right)$, $C_{1}=C_{1}\left(\alpha/2\right)$
and $C_{2}=C_{2}\left(\alpha/2\right)$. Denote by $U_{2}$ the vertices of the $\left(\frac{n'}{10},\frac{\alpha}{4}\right)$-expander in $T'_{\left[k_{1},k_{2}\right]}$,
we have $\left|U_{2}\right|\geq n'/10$. 

Let $X_{1}=X_{0}\backslash\left\{ y\right\} $, for large enough $n$ we have $\left|X_{1}\right|\geq \alpha^3n/33$. Let $\mathcal{Q}$ be a maximal set of vertex-disjoint paths in $G_1$ connecting $X_{1}$ to $U_{2}$ (meaning paths with one endpoint in $X_{1}$, one endpoint in $U_{2}$ and no internal vertices in $X_{1}\cup U_{2}$). From Lemma 2.3, $\left|\mathcal{Q}\right|\geq\frac{\alpha^{4}n}{66\left(1+\frac{\alpha}{2}\right)}\geq \alpha ^4n/100$. 
\\Set $$\mu=200/\alpha^4,$$ next we show that at least $n/\mu$ paths in $\mathcal{Q}$ contain at most $\mu$ vertices. For every positive integer $i$ denote by $\mathcal{Q}_{i}\subseteq \mathcal{Q}$ the set of all the paths in $\mathcal{Q}$ of length $i-1$ (that is with $i$ vertices). Since the paths are vertex-disjoint we have
\begin{gather*}
\mu\sum_{i>\mu}\left|\mathcal{Q}_{i}\right|<\sum_{i>\mu}i\left|\mathcal{Q}_{i}\right|\leq\sum_{i}i\left|\mathcal{Q}_{i}\right|=\left|\bigcup_{i}V\left(\mathcal{Q}_{i}\right)\right|\leq n
\end{gather*}
implying, as $\sum_{i}\left|\mathcal{Q}_{i}\right|=\left|\mathcal{Q}\right|\geq2n/\mu$, that $\sum_{i\leq\mu}\left|\mathcal{Q}_{i}\right|\geq n/\mu$. 
Denote by $X_{2}$ the set of vertices in $X_{1}$ that belong to a path of length at most $\mu-1$ in $\mathcal{Q}$, note  $\left|X_{2}\right|\geq n/\mu=\alpha^4n/200$.

For every vertex $x\in X_{2}$ denote by $Q_{x}$ the path in $\mathcal{Q}$ whose endpoint in $X_{2}$ is $x$. Given two vertices $u,w\in V(T')$ we say that $u$ eliminates $w$ if
$u$ belongs to the path from $w$ to $y$ in $T'$, that is, if $w\in V(T'_{u})$. We would like to find a path in $\left\{ Q_{x}\right\} _{x\in X_{2}}$ that eliminates a small number of vertices. 
\\For every $x\in X_{2}$ we shall denote by $b_{x}$ a vertex in $Q_{x}\cap T'$ that eliminates the most vertices from $T'$ among the vertices of $Q_{x}\cap T'$, i.e.\, 
$\left|T'_{b_{x}}\right|=\max\left\{ \left|T'_{v}\right|:v\in Q_{x}\cap T'\right\}.$
Note that such vertex exists since $Q_{x}$ has at least one vertex in $T'$ and that $Q_{x}$ eliminates at most $\left|Q_{x}\right|\cdot \left|T'_{b_{x}}\right|$ vertices.

Let $B=\left\{ b_{x}\right\}_{x\in X_{2}}$. Recall that $k_{0}$ is the first index such that  $\left|T'_{[1,k_{0}]}\right|\geq \alpha^4n/400=\frac{n}{2\mu}$.
Since $\left|B\right|=\left|X_{2}\right|\geq n/\mu$ we get that $\left|B\cap T'_{\left[k_{0},k_{2}\right]}\right|\geq\frac{n}{2\mu}$. From the pigeonhole principle, there exists an index $k_{0}\leq i\leq k_{2}$ for which the level $L_{i}$ in $T'$ contains at least 
$$\frac{\left|B\cap T'_{\left[k_{0},k_{2}\right]}\right|}{k_{2}-k_{0}+1}\geq\frac{n}{2\mu\left(C_{1}+C_{2}+1\right)}$$
vertices of $B$, denote this set by $B'$. 
Since $B'\subseteq L_{i}$, for every $b_{x_{1}},b_{x_{2}}\in B'$ the subtrees $T'_{b_{x_{1}}}$ and $T'_{b_{x_{2}}}$ are disjoint. It follows that there is a vertex $x_{0}\in X_{2}$ such that $\left|T'_{b_{x_{0}}}\right|\leq\frac{n}{|B'|}$. Otherwise we get 
$$n\geq|T'|\geq \left|\bigcup_{b_x\in B'}\left|T'_{b_{x}}\right|\right|=\sum_{b_x\in B'}\left|T'_{b_{x}}\right|>n.$$
Denote the set of all vertices in $T'$ that are eliminated by $Q_{x_{0}}$ by $Y$, observe that
$$|Y|\leq\left|Q_{x_{0}}\right|\cdot\left|T'_{b_{x_{0}}}\right|\leq \mu \cdot\frac{n}{|B'|}\leq 2\mu ^2\left(C_{1}+C_{2}+1\right)=O\left(1/\alpha^{13}\right).$$
Given a vertex $w\in L_{k_1}$, we say that the subtree $T'_w$ is good if it does not contain any vertices from $Y$, otherwise we say it is bad. We partition the vertices of $T'_{[k_1,k_2]}$ into two disjoint sets $A_G$ and $A_B$ which are the union of all good trees and the union of all bad trees in $T'_{\left[k_{1,}k_{2}\right]}$ accordingly i.e.\ 
$$
A_G=\bigcup_{\mathclap{\substack{w\in L_{k_{1}}\\
T'_{w}\cap Y=\emptyset}
}}
V\left(T'_{w}\right) \quad \text{and}\quad
A_B=\bigcup_{\mathclap{\substack{w\in L_{k_{1}}\\
T'_{w}\cap Y\neq\emptyset
}}
}V\left(T'_{w}\right)
.$$

Recall, that the degree in $T'$ of every vertex in $T'_{\left[k_{1},k_{2}\right]}$ is bounded by $\Delta$, hence for every vertex $w\in L_{k_{1}}$ we have $\left|T'_{w}\right|\leq C_{3}\coloneqq\Delta^{C_{2}+1}=2^{O\left(\frac{\log\left(1/\alpha\right)}{\alpha}\right)}$. In particular, $\left|A_B\right|\leq  C_3|Y|=2^{O\left(\frac{\log\left(1/\alpha\right)}{\alpha}\right)}$.

Let $v_{0}$ be the endpoint of $Q_{x_{0}}$ in $U_{2}$ and let $K$ be its connected component in $G_2=G\left[U_{2}\right]$. Since $G_2$ is an $\left(\frac{n'}{10},\frac{\alpha}{4}\right)$-expander, the graph $G_2[K]$ is also an $\left(\frac{n'}{10},\frac{\alpha}{4}\right)$-expander on more than $n'/10$ vertices. By applying Lemma $2.2$ to $G_2[K]$ (with $V_0=A_B$ and  $\epsilon=C_3|Y|/|K|$) we get that for large enough $n$, the graph $G_2[K\backslash A_B]$ contains a vertex set $U_{3}\subseteq K\backslash A_B=K\cap A_G$ of size $\left|U_{3}\right|\geq|K|-\frac{12C_3|Y|}{\alpha}$ such that $G_3=G\left[U_{3}\right]$ is an $\left(\frac{n'}{10},\frac{\alpha}{8}\right)$-expander. Since every vertex set in $G_3$ of size at most $n'/10$ has non-empty external neighborhood, we must have $|U_3|>n'/10$.
Let $D$ be the union of all subtrees that meet $U_3$, i.e.\, 
$$D=\bigcup_{\mathclap{\substack{w\in L_{k_{1}}\\
T'_{w}\cap U_3\neq\emptyset
}}
}V\left(T'_{w}\right).$$
Observe that since $U_{3}$ resides inside the set of good subtrees $A_G$, so does $D$. Let $Q'$ be a shortest path from $v_{0}$ to $D\cap K$ in $K$ (the set $D\cap K$ is not empty since $U_3\subseteq D\cap K$). From Lemma 2.1 (with $\delta=\frac{n'}{10|U_2|}$ and $C=K$) the length of $Q'$ is at most $\text{diam}\left(G_2[K]\right)\leq 280\log (|U_2|)/\alpha\leq 280\log n/\alpha $. Note, that $Q'$ does not contain internal vertices from $D$.

For every $w\in L_{k_{1}}\cap D$ contract $T'_{w}\cap U_{3}$ into a vertex. By Lemma $2.4$ the graph $G'$ obtained from $G_3$ by this contraction is an $\left(\frac{n'}{10C_{3}},\frac{\alpha}{8C_{3}}\right)$-expander on at least $|U_3|/C_3$ vertices. Denote by $u_0$ the endpoint of $Q'$ in $D$ and let $w_{0}\in L_{k_{1}}\cap D$ be the vertex whose subtree, $T'_{w_{0}}$ contains $u_0$. From Lemma $2.5$, for large enough $n$, $G'$ contains a path $P'$ of length at least 
$$\left\lceil \left\lfloor \frac{n'}{10C_3}\right\rfloor \cdot \frac{\alpha}{8C_3} \right\rceil\geq\frac{\alpha n'}{100C_3^2}\geq\frac{\alpha n}{10^3C_3^2}$$
starting at the vertex obtained by contracting $T'_{w_{0}}$. By reversing the contraction and adding the missing path inside each subtree $P'$ visits, we obtain a path $P$ starting at $u_0$ of length at least $\frac{\alpha n}{10^3C_3^2}$ in $D$. Moreover, for every $w\in L_{k_{1}}$the obtained path $P$ visits $T'_{w}$ at most once and $\left|P\cap T'_{w}\right|\leq2C_{2}$. 

Let $Q$ be the path from $y$ to $u_0\in P$ constructed as follows: move from $y$ to $x_0$ along $T$, then move from $x_0$ to $v_0$ along $Q_{x_0}$ and then continue along $Q'$ to $u_0$ (see Figure~2).
Denote by $m$ the length of $Q$, then 
$$m\leq\underset{y\rightarrow x_{0}}{\underbrace{\frac{30\log n}{\alpha}+\frac {12}{\alpha^2}}}+\underset{x_{0}\rightarrow v_{0}}{\underbrace{\frac{200}{\alpha^{4}}}}+\underset{v_{0}\rightarrow u_{0}}{\underbrace{\frac{280\log n}{\alpha}}}\leq\frac{10^{3}\log n}{\alpha}.$$
\\Note that the paths $Q$ and $P$ meet only in $u_0$ since all vertices in the set $\{y\}\cup V(T)\cup V(Q_{x_0})\cup V(Q')$ except $u_0$ do not belong to $D$.
\begin{figure}[h]
\includegraphics[scale=0.7]{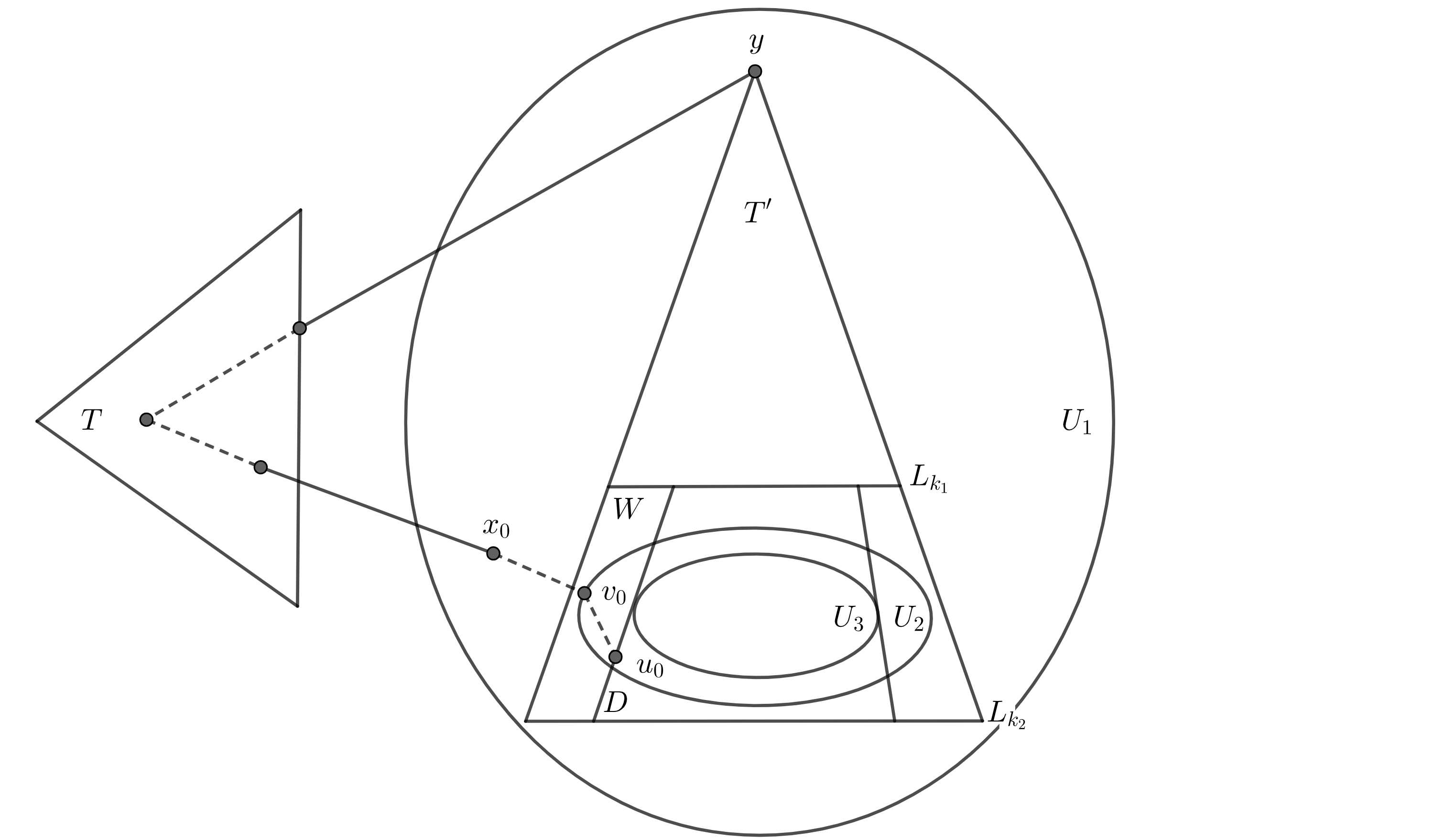}
\centering
\captionsetup{justification=centering,margin=2cm}
\caption{The construction of $Q$. \\Move from $y$ to $x_0$ along $T$ (number of steps is at most the depth of $T$), then move from $x_0$ to $v_0$ along $Q_{x_0}$ and then continue along $Q'$ to $u_0$.}
\end{figure}

Now, take $a_{1}=\frac{10^3}{\alpha}+2C_0$, $a_{2}=\frac{\alpha}{10^3C_3^2}$ and $A=3C_2$ (note that $a_1\log n \geq m+k_1$ and that $a_2n$ is at most the length of $P$). For every integer $\ell\in\left[a_{1}\log n,a_{2}n\right]$ we shall construct a cycle $C_{\ell}$ as follows (see Figure~3): 
\renewcommand\labelitemi{\tiny$\bullet$}
\begin{itemize} [topsep=0pt]
    \item Move from $y$ to $u_{0}$ along $Q$ and then continue moving along $P$ for $\ell-m-k_{1}$ steps. 
    \item Move along $P$ until leaving the current subtree $T'_{w}$ for the first time (between $1$ and $2C_{2}+1$ further steps). 
    \item Move from the current vertex to $L_{k_1}$ along $T'$ (between $0$ and $C_2$ steps), this is possible since $P$ visits every subtree in $T'_{[k_1,k_2]}$ at most once and since $V(Q)\cap D=\{u_0\}$. 
    \item Finally, move along $T'$ from the current vertex in $L_{k_1}$ (which belongs to $D$ and therefore is not eliminated by $Q_{x_0}$) to $y$ (between $k_{1}-1$ and $k_{1}+C_{2}-1$ steps). 
\end{itemize}

The length of the obtained cycle $C_{\ell}$ is between $\ell$ and $\ell+A$.
\end{proof}

\begin{figure}[h]
\centering
\includegraphics[scale=0.7]{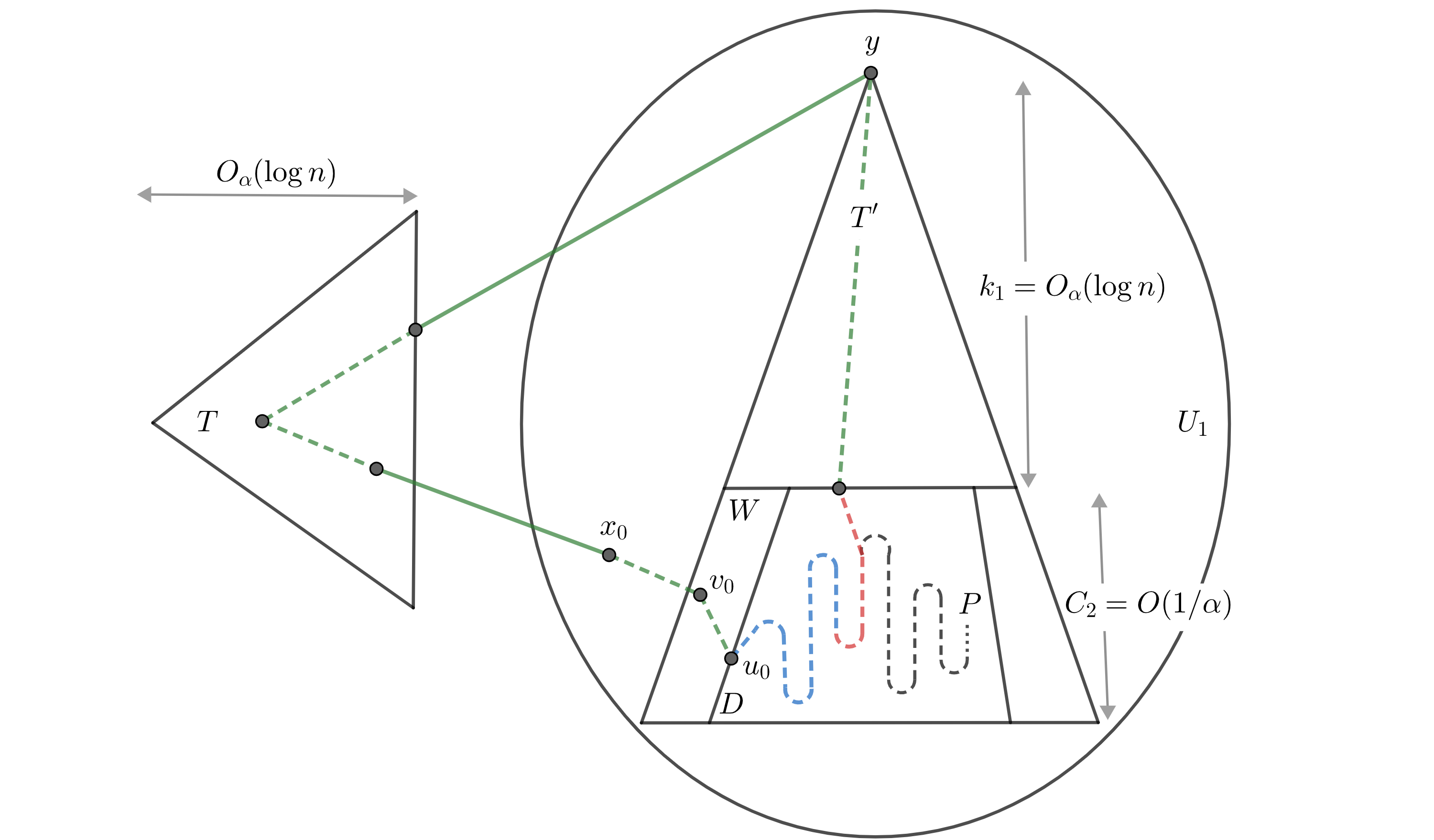}
\captionsetup{justification=centering}
\caption{The construction of $C_\ell$}
\end{figure}



\section{Number of cycle lengths is linear (better dependence on \texorpdfstring{$\alpha$)}{}}

Theorem 1 implies that for every $\alpha$-expander the number of cycle lengths is linear in the number of vertices, that is, for every $\alpha$-expander $G$ on $n$ vertices we have $\left|L\left(G\right)\right|=\Omega_{\alpha}\left(n\right)$. The hidden dependence on $\alpha$ is exponentially small; in this section we shall prove that this is rather far from the truth and that the dependence is in fact polynomial. In the proof we borrow some ideas from Sudakov and Verstraete \cite{SV}. 

\begin{customthm}{2}
Let $0<\alpha\leq1$ and let $G$ be an $\alpha$-expander on $n$ vertices, then $\left|L\left(G\right)\right|=\Omega\left(\frac{\alpha^{3}}{\log\left(1/\alpha\right)}\right)n$.
\end{customthm}

\begin{proof}
Run the BFS algorithm on $G$ using an arbitrary ordering of $V\left(G\right)$
until a tree $T$ of size exactly $\left\lfloor \alpha n/4\right\rfloor $
is obtained; this is possible since $G$ is an $\alpha$-expander and is therefore connected. Note that in
$G\backslash T$ every vertex set $W$ of size $\left|W\right|=\left\lceil n/2\right\rceil$
satisfies $$\left|N_{G\backslash T}\left(W\right)\right|\geq\left|N_{G}\left(W\right)\right|-\left|T\right|\geq\frac{\alpha n}{2}-\left\lfloor \frac{\alpha n}{4}\right\rfloor \geq\left\lceil \frac{\alpha n}{4}\right\rceil-1,$$
thus, by Lemma $2.4$, $G\backslash T$ contains a path $P$ of length
$\left\lceil \alpha n/4\right\rceil-1 $. \\
Set $$k=2\left\lceil \frac{\log\left(3/\alpha\right)}{\log\left(1+\alpha/2\right)}\right\rceil +1=O\left(\frac{\log\left(1/\alpha\right)}{\alpha}\right),$$ and let $X_{0}\subseteq V\left(P\right)$ be the set of vertices of
$P$ reachable from $T$ by a path of length at most $k$ whose internal
vertices reside all outside of $V\left(T\right)\cup V\left(P\right)$.

The set $X_{0}$ is of size at least $\alpha^{2}n/12$. Indeed, suppose
$\left|X_{0}\right|<\alpha^{2}n/12$ and look at $G'=G\left[V\left(G\right)\backslash X_{0}\right]$.
For every vertex set  $U\subseteq V\left(G'\right)$ such that $\alpha n/6\leq\left|U\right|\leq n/2$
we have $$
\left|N_{G'}\left(U\right)\right|\geq\left|N_{G}\left(U\right)\right|-\left|X_{0}\right|>\alpha\left|U\right|-\frac{\alpha}{2}\cdot\frac{\alpha n}{6}\geq\frac{\alpha}{2}\left|U\right|,$$
hence $$\left|B_{G'}\left(U,\frac{k-1}{2}\right)\right|\geq\min\left\{\frac{n}{2},\left(1+\frac{\alpha}{2}\right)^{\frac{k-1}{2}}\left|U\right|\right\}=\frac{n}{2}.$$ 
In particular, the balls $B_{G'}\left(V\left(T\right),\frac{k-1}{2}\right)$ and 
$B_{G'}\left(V\left(P\right)\backslash X_{0},\frac{k-1}{2}\right)$ each contain
at least $n/2$ vertices, therefore they intersect each other or are connected
by an edge \textemdash{} a contradiction to the way we defined $X_{0}$.

Since $T$ is a BFS tree, all these paths of length at most $k$ connecting
$T$ to $X_{0}$ start in the last two layers of $T$. Extend $T$
by adding up to $k-1$ new levels using the BFS algorithm while avoiding
adding vertices of $P$ (run the algorithm starting from the last vertex that was added to $T$ until it ends or until $k-1$ complete new levels were obtained, whichever comes first). Let
$T_{1}$ be the obtained tree. Observe that every vertex in $X_{0}$
has a neighbor in the last $k+1$ levels of $T_{1}$. At least $\frac{\left|X_{0}\right|}{k+1}$
vertices in $X_{0}$ are adjacent to the same level in $T_{1}$, w.l.o.g.
we may assume it is the last one. Denote this set of vertices by $X_{1}$,
and let $T_{2}$ be the tree obtained by connecting each vertex in
$X_{1}$ to one of its neighbors in the last level of $T_{1}$.

Let $T_{3}$ be the minimal subtree of $T_{2}$ containing $X_{1}$
and let $v$ be its root. Due to minimality $T_{3}$ branches at $v$ and each branch contains at least one vertex of $X_{1}$.
One of the branches contains at most half of the vertices of $X_{1}$,
denote this set by $Y$ and let $X_{2}=X_{1}\backslash Y$, note that $\left|X_{2}\right|\geq\left|X_{1}\right|/2$.
Let $u$ be an arbitrary vertex of $Y$; at least half of the
vertices of $X_{2}$ appear on $P$ on the same side of $u$, denote
this set by $X_{3}$ and note that 
\[
\left|X_{3}\right|\geq\frac{\left|X_{1}\right|}{4}\geq\frac{\left|X_{0}\right|}{4(k+1)}\geq\frac{\alpha^{2}n}{48\left(k+1\right)}=\Omega\left(\frac{\alpha^{3}}{\log\left(1/\alpha\right)}\right)n.
\]

Finally we construct a cycle $C_{w}$ for every $w\in X_{3}$. We get it as follows: take the path from $v$ to $u$ in $T_{3}$, then move from $u$ to $w$ along $P$ and then move back from $w$ to $v$ along $T_{3}$ (see Figure 3). One can easily see that $C_{w}$ is indeed a cycle due to the choice of $v$. As the lengths of the cycles $\left\{ C_{w}\right\} _{w\in X_{3}}$ are all distinct (if $w$ appears on $P$ earlier than $w'$ when moving from $u$, then $\left|C_{w}\right|<\left|C_{w'}\right|$), we get that $\left|L\left(G\right)\right|\geq\left|X_{3}\right|=\Omega\left(\frac{\alpha^{3}}{\log\left(1/\alpha\right)}\right)n$. 
\end{proof}
\begin{figure}[h]
\includegraphics[scale=1.4]{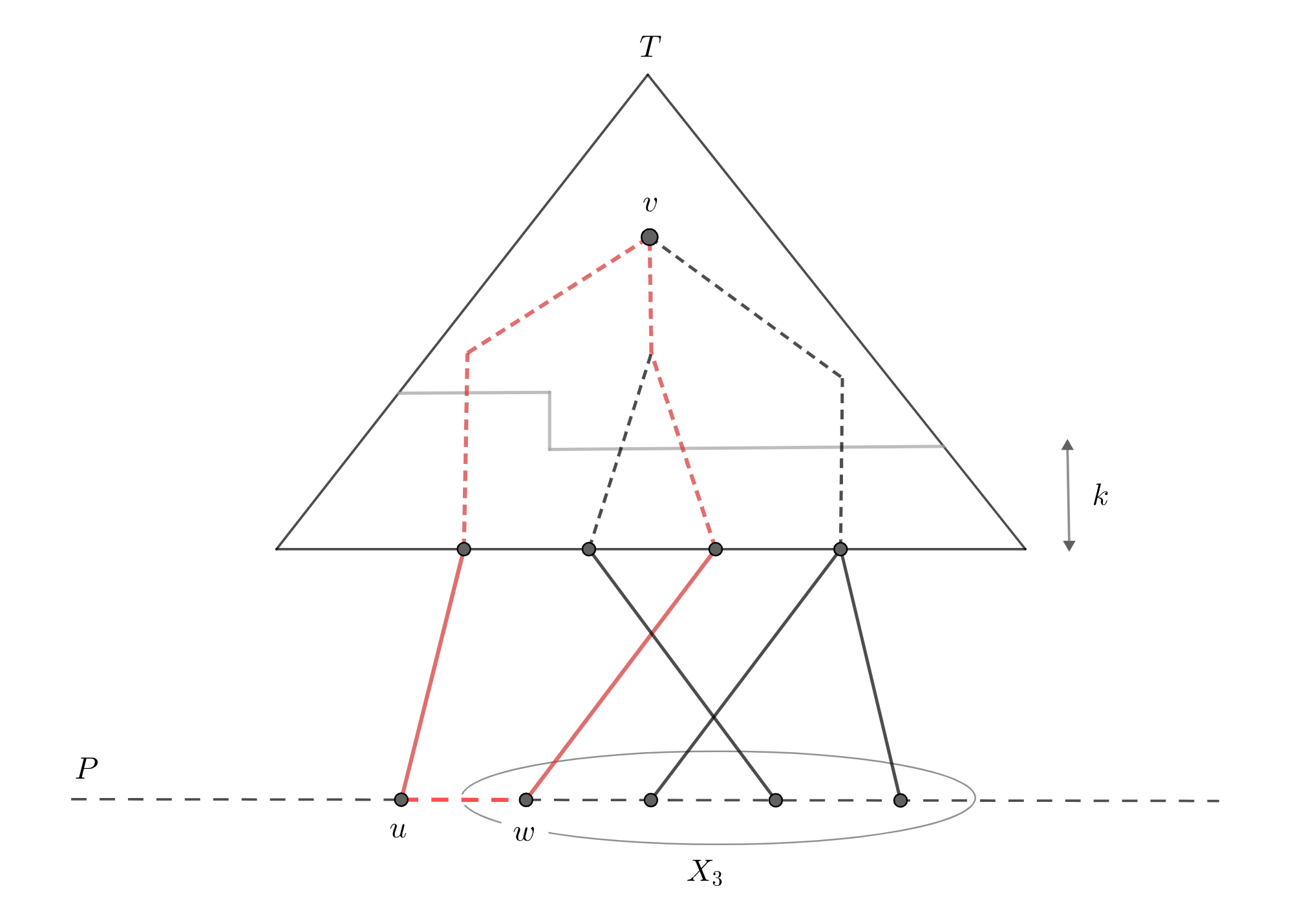}
\centering
\caption{The construction of $C_{w}$ (in red)}
\end{figure}

\section{All cycle lengths in an interval}

As mentioned in the introduction, one cannot hope to find a complete interval of non-trivial length in the set of cycle lengths of every $\alpha$-expander. In this section we present a different expansion-type property. Recall that a graph $G$ on $n$ vertices is a $\beta$-graph if every pair of disjoint vertex sets of at least $\beta n$ vertices are connected by an edge. In this section we shall prove that for every $\beta$-graph $G$ on $n$ vertices the set $L(G)$ contains a complete interval of size linear in $n$. 

\begin{customthm}{3}
For every $0<\beta<\frac{1}{20}$ there exist positive constants $b_{1}=O\left(\frac{1}{\log\left(1/\beta\right)}\right)$
and $b_{2}=O\left(\beta\right)$, s.t.\ every $\beta$-graph $G$ on $n$ vertices contains a cycle of length $\ell$ for every integer  $\ell\in\left[b_{1}\log n,(1-b_{2})n\right]$. 
\end{customthm}

In order to prove Theorem 3 we need the following lemma that shows that every $\beta$-graph on $n$ vertices contains a large $\left(\beta n,\frac{1-3\beta}{2\beta}\right)$-expander (see definition in Section 1).

\begin{lem}
Let $\beta>0$ and let $G$ be a $\beta$-graph on $n$ vertices. Then $G$ contains
a subgraph $G'$ on at least $\left(1-\beta\right)n$ vertices such
that $\left|N_{G'}\left(U\right)\right|\geq\frac{1-3\beta}{2\beta}\left|U\right|$
for every vertex set $U\subseteq V\left(G'\right)$ of size at most
$\beta n$.
\end{lem}
\begin{proof}
Start with $G'=G$, and as long as there is a vertex set $U\subseteq V\left(G'\right)$
of size $\left|U\right|\leq\beta n$ with $\left|N_{G'}\left(U\right)\right|<\frac{1-3\beta}{2\beta}\left|U\right|$,
delete $U$ and update $G'\coloneqq G'\left[V\left(G'\right)\backslash U\right]$.
\\
Let $W=V\left(G\right)\backslash V\left(G'\right)$ be the union of
the deleted sets. At any iteration, for every vertex set $U\subseteq V\left(G'\right)$
we have $N_{G}\left(W\cup U\right)\subseteq N_{G}\left(W\right)\cup N_{G'}\left(U\right)$.
Thus, by induction, $\left|N_{G}\left(W\right)\right|<\frac{1-3\beta}{2\beta}\left|W\right|$.
Assume $\left|W\right|$ reaches $\beta n$ after some iteration.
More precisely, $\beta n\leq\left|W\right|<2\beta n$, since in each
iteration we delete at most $\beta n$ vertices. We have
\[
\left|V\left(G\right)\backslash\left(W\cup N_{G}\left(W\right)\right)\right|=n-\left|W\right|-\left|N_{G}\left(W\right)\right|>n-\frac{1-\beta}{2\beta}|W|>\beta n,
\]
implying there is an edge between $W$ and $V\left(G\right)\backslash\left(W\cup N_{G}\left(W\right)\right)$
\textemdash{} a contradiction. 
\end{proof}

The main tool used in the proof of Theorem 3 is the following powerful generalisation of a result by Friedman and Pippenger \cite{FP} due to Haxell \cite{H}, which implies that strong expanders contain nearly all spanning bounded degree trees. We state a version as given in \cite{BCPS}.
\begin{customthm}{4.1}
Let $d$, $m$ and $M$ be positive integers. Assume that $H$ is a non-empty graph satisfying the following two conditions:

1. For every $U\subseteq V\left(H\right)$ with $0<\left|U\right|\leq m$,
$\left|N_{H}\left(U\right)\right|\geq d\left|U\right|+1$.

2. For every $U\subseteq V\left(H\right)$ with $m<\left|U\right|\leq2m$,
$\left|N_{H}\left(U\right)\right|\geq d\left|U\right|+M$.

Then $H$ contains every tree $T$ with $M$ vertices and maximum degree at most $d$.
\end{customthm}

\noindent We are now ready to prove Theorem 3:

\begin{proof}
Given positive integers $k$, $t$ and $p$, denote by $T_{k,t,p}$ the tree obtained by joining the roots of two complete $k$-ary trees of depth $t$ by a path of length $p$ (by a complete $k$-ary tree of depth $t$ we mean a rooted tree whose internal vertices have $k$ children and whose leaves have distance $t$ from the root). Observe that if $G$ contains a copy of $T_{k,t,\ell-2t-1}$ where $k\geq2$ and $t\geq\log_{k}\left(\beta n\right)$, then it also contains a cycle of length $\ell$. Indeed, for such $k$ and $t$, a $k$-ary tree of depth $t$ has at least $\beta n$ leaves. In particular, if $G$ contains a copy of $T_{k,t,\ell-2t-1}$, since it is a $\beta$-graph there must be an edge between the leaves of the two $k$-ary trees in this copy and this edge closes a cycle of length $\ell$ (see Figure 4).

From Lemma $4.1$, $G$ contains an induced subgraph $H$ on at least $\left(1-\beta\right)n$ vertices, such that for every vertex set $U\subseteq V\left(H\right)$ of size $\left|U\right|\leq\beta n$  we have $\left|N_{H}\left(U\right)\right|\geq\frac{1-3\beta}{2\beta}\left|U\right|$. Moreover, if $\left|U\right|\geq\beta n$ then we must have $\left|V\left(H\right)\backslash\left(U\cup N_{H}\left(U\right)\right)\right|<\beta n$ as otherwise, since $G$ is a $\beta$-graph we get that there is an edge between $U$ and $V\left(H\right)\backslash\left(U\cup N_{H}\left(U\right)\right)$.
In particular, for every vertex set $U\subseteq V\left(H\right)$
of size $\beta n\leq\left|U\right|\leq2\beta n$ we get that $$\left|N_{H}\left(U\right)\right|>\left|V\left(H\right)\right|-\left|U\right|-\beta n\geq\left(1-4\beta\right)n.$$
These properties of $H$ together with Theorem~4.1 guarantee it contains a copy of:
\begin{enumerate}
    \item $T_{k,t,\ell-2t-1}$ for $k=\left\lfloor \frac{1}{4\beta}\right\rfloor$,
    $t=\left\lceil\log_{k}\left(\beta n\right)\right\rceil$,
    and every integer $\ell\in\left[2t+1,\left(\frac{1}{2}-10\beta\right)n\right]$.
    \item $T_{2,r,\ell-2r-1}$ for $r=\left\lceil\log\left(\beta                n\right)\right\rceil$ and every integer                                $\ell\in\left[2r+1,\left(1-14\beta\right)n\right]$.
\end{enumerate}

Now, set $b_{1}=25/\log(1/\beta)$ and $b_{2}=1-14\beta$. Note that for $\beta<1/20$ we have $2\left\lceil \log\left(\beta n\right)\right\rceil +1<\left(1/2-10\beta\right)n$ and $b_{1}\log n\geq2t+1$, therefore we get that $H$ contains a cycle of length $\ell$ for every integer $b_{1}\log n\leq\ell\leq b_{2}n$.
\\

\begin{center}


    \includegraphics[scale=1.1]{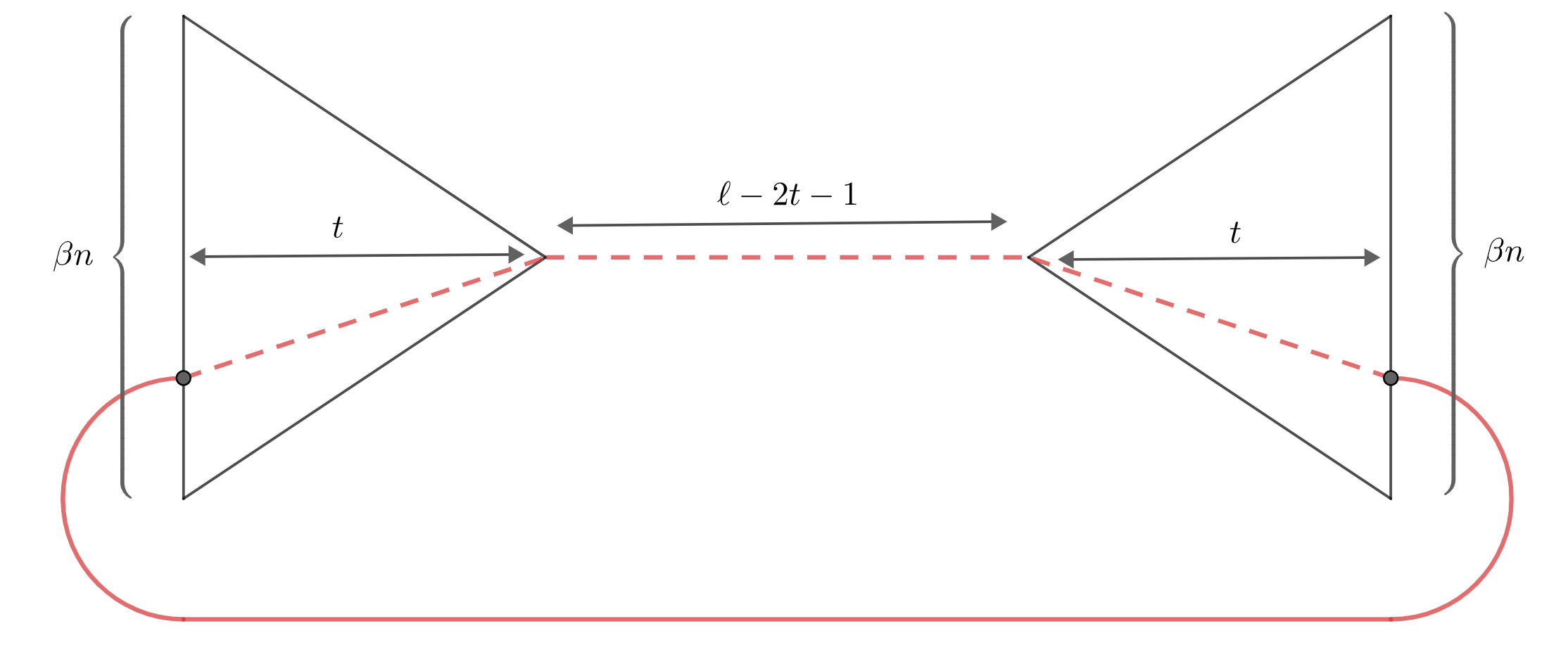}
	\captionof{figure}{Construction of a cycle of length $\ell$ (in red) using a copy of $T_{k,t,\ell-2t-1}$}

\end{center}

We are left to verify claims $\left(1\right)$ and $\left(2\right)$
above. This is quite straightforward:
\begin{enumerate}
    \item We apply Theorem~4.1 with $m=\beta n$, $d=k+1$,
$M=\ell+2\frac{k\beta n-1}{k-1}-1\leq\ell+4\beta n$.
\\For every $U\subseteq V\left(H\right)$ of size $0<\left|U\right|\leq\beta n$
we have 
\[
\left|N_{H}\left(U\right)\right|\ge\frac{1-3\beta}{2\beta}\left|U\right|=\left(\frac{1}{4\beta}+1\right)\left|U\right|+\underset{\text{\ensuremath{\geq1}}}{\underbrace{\frac{1-10\beta}{4\beta}}}
 \left|U\right|\geq d\left|U\right|+1.
\]
For every $U\subseteq V\left(H\right)$ of size $\beta n\leq\left|U\right|\leq2\beta n$
we have 
\[
\left|N_{H}\left(U\right)\right|>\left(1-4\beta\right)n=\left(\frac{1}{4\beta}+1\right)2\beta n+\left(\frac{1}{2}-10\beta\right)n+4\beta n\geq d\left|U\right|+\ell+4\beta n\geq d\left|U\right|+M.
\]
\item We apply Theorem~4.1 with $m=\beta n$, $d=3$, $M=\ell+4\beta n-3$. \\
For every $U\subseteq V\left(H\right)$ of size $0<\left|U\right|\leq\beta n$
we have
\[
\left|N_{H}\left(U\right)\right|\geq\frac{1-3\beta}{2\beta}\left|U\right|=3\left|U\right|+\underset{\text{\ensuremath{\geq1}}}{\underbrace{\frac{1-9\beta}{2\beta}}}
 \left|U\right|\geq d\left|U\right|+1.
\]
For every $U\subseteq V\left(H\right)$ of size $\beta n\leq\left|U\right|\leq2\beta n$
we have
\[
\left|N_{H}\left(U\right)\right|>\left(1-4\beta\right)n=6\beta n+(1-14\beta)n+4\beta n\geq3\left|U\right|+\ell+4\beta n\geq d\left|U\right|+M.
\]
\end{enumerate}
\end{proof}

\section{Concluding remarks}
In this work we studied cycle lengths in expanding graphs. As it turns out, even the relatively weak quantitative notion of expanders we adopted in this paper guarantees several meaningful and interesting results about cycle lengths. We have shown that the number of cycle lengths in $\alpha$-expanders is linear in the number of vertices and that the set of cycle lengths is well distributed. We also introduced another expansion-type property, guaranteeing the existence of a linearly long interval in the set of cycle lengths.

In Theorem~1 we proved that for every $0<\alpha\leq1$ there exist positive constants $n_{0}$, $C$ and $A=O(1/\alpha)$ such that for every $\alpha$-expander $G$ on $n\geq n_{0}$ vertices and every integer $\ell\in\left[C\log n,\frac{n}{C}\right]$, $G$ contains a cycle whose length is between $\ell$ and $\ell+A$. The order of dependence of the additive error term $A$ on $\alpha$ is optimal. However, the constant $C$ obtained in our proof is exponentially large in $\alpha$. We believe it can be improved to polynomial in $\alpha$.

In Theorem~2 we showed that every $\alpha$-expander on $n$ vertices contains $\Omega\left(\frac{\alpha^{3}}{\log\left(1/\alpha\right)}\right)n$ different cycle lengths. We conjecture that this result can be further improved to a linear dependence in $\alpha$.

Finally, in Theorem~3 we proved that for every $\beta <1/20$ there exist positive constants $b_{1}=O\left(\frac{1}{\log\left(1/\beta\right)}\right)$ and $b_{2}=O\left(\beta\right)$ such that every $\beta$-graph $G$ on $n$ vertices contains a cycle of length $\ell$ for every integer $\ell\in\left[b_{1}\log n,(1-b_{2})n\right]$; the order of dependence of $b_{1}$ and $b_{2}$ on $\beta$ is optimal. This answers and improves a conjecture by Hefetz, Krivelevich and  Szab\'o \cite{HKS}.

Another well studied circle of problems is the modular arithmetic of cycle lengths in graphs with certain properties (see e.g. \cite{AL,Bo,CS,DLS,E,F,SV2,T,V}). To the best of our knowledge this problem has not been researched explicitly in the setting of expanders. For example, given a positive integer $q$, can one guarantee the existence of a cycle whose length is divisible by $q$ in every $\alpha$-expander?

\vspace{4mm}
\noindent\textbf{Acknowledgements.} The authors are grateful to Rajko Nenadov for his contribution to this paper. The authors also wish to thank Asaf Cohen, Wojciech Samotij and Leonid Vishnevsky for their input and remarks.



\begin{thebibliography}{99}
 \doublespacing 

\bibitem{A}
N. Alon, 
{\it  Eigenvalues and expanders}, Combinatorica 6, 83--96,  1986.
 
\bibitem{AL}
N. Alon and N. Linial,
{\it Cycles of length 0 modulo k in directed graphs}, Journal of Combinatorial Theory, Ser. B 47, 114--119, 1989.

\bibitem{AS}
N.\ Alon and J. H. Spencer, {\bf The Probabilistic Method}, 4th edition, Wiley, New York, 2015.

\bibitem{BCPS}
J. Balogh, B. Csaba, M. Pei, and W. Samotij, 
{\it Large bounded degree trees in expanding graphs}, Electronic Journal of Combinatorics
17, Research Paper 6, 2010.

\bibitem{Bo}
B. Bollob{\'a}s, 
{\it Cycles modulo k}, Bulletin of the London Mathematical Society 9, 97--98, 1977.

\bibitem{B}
J.A. Bondy, 
{\it Pancyclic graphs I}, Journal of Combinatorial Theory, Ser. B 11, 80--84, 1971.


\bibitem{CS}
G.T. Chen and A. Saito,
{\it Graphs with a Cycle of Length Divisible by Three}, Journal of Combinatorial Theory, Ser. B 60, 277--292, 1994.

\bibitem{DLS}
N. Dean, L. Lesniak and A. Saito,
{\it Cycles of length 0 modulo 4 in graphs}, Discrete Mathematics
121, 37--49, 1993.

\bibitem{E}
P. Erd\H{o}s,
{\it Some recent problems and results in graph theory, combinatorics and number theory}, Proceedings of the 7th Southeast Conference on Combinatorics, Graph Theory and Computing, 3--14, 1976.

\bibitem{EFRS}
P. Erd\H{o}s, R. Faudree, C. Rousseau and R. Schelp, 
{\it The number of cycle lengths in graphs of given minimum degree and girth}, Discrete Mathematics 200, 55--60, 1999.

\bibitem{F}
G. Fan, 
{\it Distribution of Cycle Lengths in Graphs}, Journal of Combinatorial Theory, Ser. B 82, 187--202, 2002.


\bibitem{FP}
J. Friedman and N. Pippenger, {\it Expanding graphs contain all small trees}, Combinatorica 7, 71--76,  1987.


\bibitem{GHLM}
J. Gao, Q. Huo, C. H. Liu and J. Ma, 
{\it A unified proof of conjectures on cycle lengths in graphs}, arXiv:1904.08126 [math.CO], 2019.


\bibitem{GKS}
A. Gy\'arf\'as, J. Koml\'os and E. Szemer\'edi, {\it On the distribution
of cycle lengths in graphs}, Journal of Graph Theory 8, 441--462, 1984.

\bibitem{H}
P. E. Haxell. {\it Tree embeddings}, Journal of Graph Theory 36, 121--130, 2001.

\bibitem{HKS}
D. Hefetz, M. Krivelevich and T. Szab\'o, 
{\it Hamilton cycles in highly connected and expanding graphs}, Combinatorica 29, 547--568, 2009.

\bibitem{HLW}
S. Hoory, N. Linial and A. Wigderson, 
{\it Expander graphs and their applications}, Bulletin of the American Mathematical Society 43, 439--561, 2006.

\bibitem{KSV}
A. Kostochka, B. Sudakov and J. Verstra\"ete, 
{\it Cycles in triangle-free graphs of large chromatic number}, Combinatorica 37, 481--494, 2017.

\bibitem{Kri}
M. Krivelevich, 
{\it Long paths and Hamiltonicity in random graphs}, Random Graphs, Geometry and Asymptotic Structure, London Mathematical Society Student Texts 84, Cambridge
University Press, 4--27, 2016.

\bibitem{Kr}
M. Krivelevich, 
{\it Finding and using expanders in locally sparse graphs}, SIAM Journal on Discrete Mathematics 32, 611--623, 2018.

\bibitem{K}
M. Krivelevich, 
{\it Expanders --- how to find them, and what to find in them},  in Surveys in Combinatorics 2019, A. Lo et al., Eds., London Mathematical Society Lecture Notes 456, 115--142, 2019.



\bibitem{LM}
C.H. Liu and J. Ma, 
{\it Cycle lengths and minimum degree of graphs},  Journal of Combinatorial Theory, Ser. B 128, 66--95, 2018.

\bibitem{LPS}
A. Lubotzky, R. Phillips and P. Sarnak, 
{\it Ramanujan graphs},
Combinatorica 8, 261--277, 1988.

\bibitem{MS}
P. Mih\'ok and I. Schiermeyer, 
{\it Cycle lengths and chromatic number of graphs}, Discrete Mathematics 286, 147--149, 2004.

\bibitem{SV}
B. Sudakov and J. Verstra\"ete, 
{\it Cycle lengths in sparse graphs}, Combinatorica 28, 357--372, 2008.

\bibitem{SV2}
B. Sudakov and J. Verstra\"ete, 
{\it The Extremal Function for Cycles of Length $\ell$ mod $k$}, Electronic Journal of Combinatorics 24, Paper 1.7, 2016.

\bibitem{T}
C. Thomassen,
{\it Graph decomposition with applications to subdivisions and path systems modulo k}, Journal of Graph Theory 7, 261--271, 1983.

\bibitem{V}
J. Verstra\"ete, 
{\it On arithmetic progressions of cycle lengths in graphs}, Combinatorics, Probability and Computing 9, 369--373, 2000.
















\end{thebibliography}
\end{document}